\documentclass[leqno,3p]{elsarticle}
\journal{Journal of Pure and Applied Algebra}

\usepackage{ulem}
\usepackage[english]{babel}
\usepackage{lmodern}
\usepackage{amsmath,amssymb,amsthm}
\usepackage{eqparbox}
\usepackage{graphicx}
\usepackage{float}
\usepackage{etoolbox}
\patchcmd{\thmhead}{(#3)}{#3}{}{} % remove parentheses in optional argument of theorem

\usepackage{enumitem}
\setlist[enumerate]{label=(\textit{\roman*}\hspace{.08em}),font=\rm,itemsep=0em}

\usepackage[colorlinks]{hyperref}
\hypersetup{
  pdftitle   = {Classical deformations of noncompact surfaces and their moduli of instantons},
  pdfauthor  = {Severin Barmeier, Elizabeth Gasparim},
  pdfcreator = {}
}
\usepackage{tikz}
\usetikzlibrary{matrix,arrows,decorations.markings}
\newcommand{\totikz}{\mathrel{\tikz[baseline]\draw[ ->,line width=.45pt] (0ex,0.65ex) -- (3ex,0.65ex);}}
\newcommand{\mapstotikz}{\mathrel{\tikz[baseline]\draw[|->,line width=.45pt] (0pt,0.65ex) -- (3ex,0.65ex);}}
\newcommand{\rightarrowtikz}{\mathrel{\tikz[baseline]\draw[->,line width=.45pt] (0ex,0.65ex) -- (3ex,0.65ex);}}
\newcommand{\hookrightarrowtikz}{\mathrel{\tikz[baseline]\draw[right hook->,line width=.45pt] (0ex,0.65ex) -- (3ex,0.65ex);}}
\newcommand{\leftrightarrowtikz}{\mathrel{\tikz[baseline]\draw[<->,line width=.45pt] (0ex,0.65ex) -- (3.5ex,0.65ex);}}
\newcommand{\toarg}[1]{\mathrel{\tikz[baseline]\path[->,line width=.45pt] (0ex,0.65ex) edge node[above=-.4ex, overlay, font=\scriptsize] {$#1$} (3.5ex,.65ex);}}
\newcommand{\wtilde}[1]{\mkern1.75mu\widetilde{\mkern-1.75mu #1\mkern-.25mu}\mkern.25mu}

\numberwithin{equation}{section}
\newtheorem{theorem}[equation]{Theorem}
\newtheorem*{theorem*}{Theorem}
\newtheorem{proposition}[equation]{Proposition}
\newtheorem{lemma}[equation]{Lemma}
\newtheorem*{lemma*}{Lemma}
\newtheorem{corollary}[equation]{Corollary}
\newtheorem*{corollary*}{Corollary}
\theoremstyle{definition}
\newtheorem{definition}[equation]{Definition}

\newtheorem{notation}[equation]{Notation}
\theoremstyle{remark}
\newtheorem{remark}[equation]{Remark}
\DeclareMathOperator{\Ext}{Ext}

\DeclareMathOperator{\Pic}{Pic}
\DeclareMathOperator{\Tot}{Tot}
\DeclareMathOperator{\im}{im}

\DeclareMathOperator{\coim}{coim}
\DeclareMathOperator{\Coh}{Coh}
\renewcommand{\H}{\operatorname{H}}

\newcommand{\overbar}[1]{{\mkern3.5mu\overline{\mkern-3.5mu #1\mkern-.5mu}\mkern.5mu}}

\begin{document}

\begin{frontmatter}

\title{Classical deformations of noncompact surfaces and their moduli of instantons}

\author[wwu]{Severin Barmeier}
\ead{s.barmeier@gmail.com}
\author[ucn]{Elizabeth Gasparim}
\ead{etgasparim@gmail.com}

\address[wwu]{Westf\"alische Wilhelms-Universit\"at M\"unster, Mathematisches Institut, Einsteinstr.\ 62, M\"unster, Germany}
\address[ucn]{Departamento de Matem\'aticas, Universidad Cat\'olica del Norte, Av.\ Angamos 0600, Antofagasta, Chile}

\begin{abstract}
We describe semiuniversal deformation spaces for the noncompact surfaces $Z_k := \Tot (\mathcal O_{\mathbb P^1} (-k))$ and prove that any nontrivial deformation $Z_k (\tau)$ of $Z_k$ is affine.

It is known that the moduli spaces of instantons of charge $j$ on $Z_k$ are quasi-projective varieties of dimension $2j-k-2$. In contrast, our results imply that the moduli spaces of instantons on any nontrivial deformation $Z_k (\tau)$ are empty.
\end{abstract}


\end{frontmatter}

\section{Motivation}
Our interest in deformations of noncompact surfaces and their moduli of vector bundles arose in an attempt to understand how instanton moduli vary in families.

Theories of instantons and their moduli spaces are often defined over noncompact manifolds, as is the case with the instanton partition function, defined by Nekrasov \cite{Ne} and explored by various authors, for instance \cite{NO,NY,GL,BE}.

We study how moduli spaces of instantons on the noncompact surfaces $Z_k = \Tot (\mathcal O_{\mathbb P^1} (-k))$ behave under deformation of the complex structure of the underlying surface. We construct a family of deformations for these surfaces and then study holomorphic vector bundles on the deformed surfaces. The Kobayashi--Hitchin correspondence associates instantons on complex surfaces to holomorphic vector bundles. Therefore, we describe holomorphic bundles on the deformed surfaces to obtain the corresponding information about instantons. The Kobayashi--Hitchin correspondence for the surfaces $Z_k$ was shown in \cite[Prop.\ 5.3]{GKM}; the proof uses compactification and appeals to the compact version of the correspondence as described in \cite{LT}.

In contrast to the deformation theory of {\it compact} complex manifolds as developed by Kodaira and Spencer \cite{kodaira}, a general deformation theory for {\it noncompact} complex manifolds has yet to be developed. Nevertheless, under certain additional assumptions the deformation theory of noncompact complex manifolds seems to be well-behaved: for example, in case the manifold admits a global holomorphic symplectic form \cite{kaledinverbitsky}, or when the manifold compactifies holomorphically, in which case the machinery for compact manifolds may be applied \cite{kawamata}. In \S \ref{classicaldeformations} we exploit the vector bundle structure on $Z_k$ and consider deformations as an affine bundle over $\mathbb P^1$. Relations to the deformations of Hirzebruch surfaces are given in \S \ref{relationtohirzebruch}.

The noncompact surfaces $Z_k$ admit a rich structure of moduli spaces of instantons. 
Some properties of these moduli are described in \cite{GKM}, where it is shown that such moduli spaces are quasi-projective varieties whose dimensions increase with the charge. Here we show that after a deformation of the complex structure, the moduli of instantons on the deformation are empty, {\it i.e.}\ instantons disappear after a ``classical'' deformation of $Z_k$. From the point of view of mathematical physics, our results suggest that to study the instanton moduli in families, deformations should be considered in a broader framework, including {\it noncommutative} deformations. In a future paper we will pursue this more general approach.

Even though the whole story for our original motivation has yet to be told, we decided to publish the results pertaining to classical deformations separately, because these are -- to the best of our knowledge -- the first results in this direction and, moreover, are turned out to be of independent interest: deformations of noncompact surfaces (quite mysteriously) appear in an unrelated body of work concerning the homological mirror symmetry conjecture from a Lie theoretical viewpoint (see \cite{BBGGS} and Remark \ref{mirrorsymmetry}).

The paper is organized as follows: In \S\S \ref{localsurfaces}--\ref{geometry} we introduce the noncompact surfaces $Z_k$ and their moduli of vector bundles. In \S \ref{classicaldeformations} we present deformations of these surfaces and study their moduli of vector bundles in \S \ref{geometrydeformations}. Applications to the theory of instantons are discussed in \S \ref{instantons}.

\section{Statement of results}

We consider noncompact surfaces that are total spaces of negative line bundles on the projective line, denoted by $Z_k := \Tot (\mathcal O_{\mathbb P^1} (-k))$, for $k \geq 1$.
Our first result (Theorem~\ref{family}) shows that $Z_k$ admits a $(k{-}1)$-dimen\-sional semiuniversal family of classical deformations, and we construct this family explicitly. Denoting by $Z_k (\tau)$ any nontrivial deformation of $Z_k$ for $k \geq 2$, our second result (Theorem~\ref{nocompactcurves}) shows that $Z_k (\tau)$ contains no compact complex analytic curve. 
Our third result (Theorem~\ref{decomposable}) shows that any holomorphic vector bundle 
on $Z_k (\tau)$ splits as a direct sum of algebraic line bundles. This is somewhat surprising, given the existence of nontrivial moduli of vector bundles on the original $Z_k$ surfaces proved in \cite{ballicogasparimkoppe2}. Our fourth result (Theorem~\ref{affine}) shows that any nontrivial deformation $Z_k (\tau)$ is affine.

These results imply that moduli of instantons on noncompact surfaces are sensitive to the complex structure: the instanton moduli of charge $j$ over the noncompact surfaces $Z_k$ are of dimension $2j - k - 2$, whereas $Z_k (\tau)$ admits no instantons (Theorem~\ref{empty}).

Let us put our results also into the context of deformations of curves and surfaces and their moduli. Grothendieck's splitting theorem says that any holomorphic vector bundle on $\mathbb P^1$ splits as a direct sum of (algebraic) line bundles. Neither the curve $\mathbb P^1$ itself nor its moduli spaces of vector bundles admit any deformations.

Curves of higher genus do admit deformations and a celebrated theorem of Narasimhan and Ramanan \cite{narasimhanramanan,balajivishwanath} shows that {\it all} classical deformations 
of the moduli of stable bundles on a smooth curve come from deformations of the curve itself (case $g > 1$, $(r, d) = 1$).

In higher dimensions deformations of the underlying space and deformations of sheaves may be related as follows. A pair $(X, \mathcal E)$ of a smooth projective surface $X$ and a, say, (semi)stable coherent sheaf $\mathcal E$ can be thought of as a point in the moduli space of semistable coherent sheaves over $X$. Deformations of the pair $(X, \mathcal E)$ are parametrized by a certain sheaf of differential graded Lie algebras \cite{iaconomanetti} whose cohomology groups $\mathrm T^i_{(X, \mathcal E)}$ fit into a long exact sequence
\begin{align}
\label{pairsexact}
\dotsb \totikz \Ext^i_X (\mathcal E, \mathcal E) \totikz \mathrm T^i_{(X, \mathcal E)} \totikz \H^i (X, \mathcal T_X) \totikz \Ext^{i+1}_X (\mathcal E, \mathcal E) \totikz \dotsb
\end{align}
If $\mathcal E$ is a smooth point of the moduli space, {\it i.e.}\ if $\Ext^2_X (\mathcal E, \mathcal E) = 0$, then (\ref{pairsexact}) gives a surjection $\mathrm T^1_{(X, \mathcal E)} \totikz \H^1 (X, \mathcal T_X)$ of tangent spaces and an injection $\mathrm T^2_{(X, \mathcal E)} \totikz \H^2 (X, \mathcal T_X)$ of obstruction spaces. A cocycle $\tau \in \H^1 (X, \mathcal T_X)$ parametrizing a deformation of $X$ thus lifts to a deformation of the pair $(X, \mathcal E)$.

From this point of view, our results for moduli of vector bundles on nontrivial deformations of $Z_k$ show that the deformed sheaf $\mathcal E (\tau)$ over a nontrivial deformation $Z_k (\tau)$ is rigid, in spite of the fact that $\mathcal E$ over $Z_k$ is not rigid. In fact, \cite[Thm.\ 4.11]{ballicogasparimkoppe2} shows that the moduli of framed-stable rank~$2$ bundles on $Z_k$ with {\it splitting type} $j$ (see Def.\ \ref{splittingtype}) and $c_1 = 0$ are quasi-projective varieties of dimension $2j-k-2$. (See \cite{ballicogasparim} for arbitrary $c_1$.)

\section{Noncompact surfaces}
\label{localsurfaces}

Let $Z_k$ be the total space of the line bundle $\mathcal O_{\mathbb P^1} (-k)$ for $k \geq 1$. We observe that $Z_1$ is $\widetilde {\mathbb C}^2$, the blowup of $\mathbb C^2$ at the origin, and $Z_2$ is a {\it local Calabi--Yau surface}.\footnote{In this terminology, motivated by physics, a {\it local Calabi--Yau} refers to the total space of a canonical bundle, see {\it e.g.}\ \cite{Br}. Informally, we sometimes refer to $Z_k$ as a {\it local surface} as in \cite{amilburubarmeiercallandergasparim,benbassatgasparim}. However, this terminology differs from the concept of local in the sense of germs such as in \cite{GR}.}

Our main objects of study will be classical deformations of the surfaces $Z_k$ and their moduli spaces of vector bundles. 

\begin{remark}
In this work we restrict our study to $Z_k$ for $k > 0$, in which case holomorphic bundles are algebraic (Thm.\ \ref{filtrable}). As a consequence moduli spaces of vector bundles over $Z_k$ are finite dimensional.
\end{remark}

\begin{notation}
We fix once and for all coordinate charts on $Z_k$, which we refer to as {\it canonical coordinates}, given by
\begin{equation*}
U = \mathbb C^2_{z,u} = \bigl\{(z,u)\bigr\} \qquad\text{and}\qquad
V = \mathbb C^2_{\xi,v} = \bigl\{(\xi, v)\bigr\} \text{,}
\end{equation*}
such that on $U \cap V = \mathbb C^* \times \mathbb C$ we identify
\begin{align}
\label{identification}
(\xi, v) = (z^{-1}, z^k u)\text{.}
\end{align}
\end{notation}

\begin{remark}
The cover $\{ U, V \}$ is a Leray cover and throughout we will calculate sheaf cohomology as \v Cech cohomology with respect to the cover $\{ U, V \}$.
\end{remark}

\section{Geometry and topology of $Z_k$}
\label{geometry}

\subsection{Line bundles on $Z_k$} 
We have $\H^1 (Z_k, \mathcal O_{Z_k}) = \H^2 (Z_k, \mathcal O_{Z_k}) = 0$. Thus, the exponential sheaf sequence
\[
0 \totikz \mathbb Z \totikz \mathcal O \toarg{\exp} \mathcal O^* \totikz 0
\]
 implies that $\Pic Z_k \simeq \H^1 (Z_k, \mathcal O^*) \simeq \H^2 (Z_k, \mathbb Z) \simeq \mathbb Z$, whence line bundles on $Z_k$ are determined by their first Chern class. We write $\mathcal O_{Z_k} (n)$, or simply $\mathcal O (n)$,
for the line bundle with first Chern class $n$. In canonical coordinate charts the bundle $\mathcal O(n)$
has the transition matrix $\bigl(z^{-n}\bigr)$.

\subsection{Vector bundles on $Z_k$}

Recall that a rank~$r$ bundle $E$ over a variety $X$ is called {\it filtrable} if there exists an increasing filtration $0 = E_0 \subset E_1 \subset \dotsb \subset E_{r-1} \subset E_r = E$ of subbundles such that $E_i / E_{i-1} \in \Pic X$, where $1 \leq i \leq r$. We now recall some properties of vector bundles on $Z_k$.

\begin{theorem}[{\cite[Lem.\ 3.1, Thm.\ 3.2]{gasparim1}}]
\label{filtrable}
Holomorphic vector bundles on $Z_k$ are algebraic and filtrable.
\end{theorem}

In particular, any rank~$2$ bundle on $Z_k$ is isomorphic to an algebraic extension of line bundles. 
Theorem \ref{filtrable} generalizes to the case of ample conormal bundle, see 
\cite[Thm.\ 3.2]{ballicogasparimkoppe1}.

\begin{notation}
\label{isomorphism}
Given two vector bundles $E$ and $E'$ over $Z_k$, defined by transition matrices $T$ and $T'$ respectively, a vector bundle {\it isomorphism} $E \simeq E'$ is given by a pair of invertible matrices $(A_U, A_V)$, where $A_U$ (resp.\ $A_V$) and its inverse have entries holomorphic in $U$ (resp.\ $V$) and such that $A_V T A_U = T'$. In particular, $\det A_U = \det A_V \in \mathbb C^*$.
\end{notation}

We shall make repeated use of the following standard result.
\begin{lemma}[{\cite[III.6.3.(c), III.6.7]{hartshorne}}]
\label{extensioncohomology}
There exists an isomorphism
\[
\Ext^1 (\mathcal O_{Z_k} (j), \mathcal O_{Z_k} (-j)) \simeq \H^1 (Z_k, \mathcal O_{Z_k} (-2j))\text{.}
\]
\end{lemma}

In terms of canonical coordinates on $Z_k$, the class of an extension defined by the transition matrix
\[
\begin{pmatrix}
z^j & p \\
0   & z^{-j}
\end{pmatrix}
\]
is sent to the cohomology class represented by the $1$-cocycle $z^{-j} p$, where $p$ and the $1$-cocycle $z^{-j} p$ are written in $U$-coordinates.

A rank~$2$ bundle with $c_1 = 0$ may thus be given by a cohomology class $\sigma \in \H^1 (Z_k, \mathcal O (-2j))$ whose general form we recall in the following lemma.

\begin{lemma}[{\cite[Lem.\ 2.6]{ballicogasparimkoppe2}}]
\label{tang}
Set $m = \big\lfloor \tfrac{n-2}k \big\rfloor$. A cohomology class in $\H^1 (Z_k, \mathcal O (-n))$ is represented by a $1$-cocycle of the general form 
\begin{align}
\label{generalcocycle}
\sigma = \sum_{i=0}^m \sum_{l=ik-n+1}^{-1} \sigma_{il}z^l u^i, \qquad \sigma_{il} \in \mathbb C.
\end{align}
In particular, we have that
\begin{align*}
\dim \H^1 (Z_k, \mathcal O (-n)) = \frac{(m+1)(2n-km-2)}2 \qquad\text{if $n \geq 2$}
\end{align*}
and zero otherwise.
\end{lemma}

\begin{definition}[\cite{ballico}]
\label{splittingtype}
Let $E$ be a rank~$r$ bundle on $Z_k$. Then the restriction of $E$ to the zero section $\ell \simeq \mathbb P^1$ is a rank~$r$ bundle on $\mathbb P^1$, which by Grothendieck's splitting theorem splits as a direct sum of line bundles. That is, $E \vert_\ell \simeq \mathcal O_{\mathbb P^1} (j_1) \oplus \dotsb \oplus \mathcal O_{\mathbb P^1} (j_r)$. We call $(j_1, \dotsc, j_r)$ the {\it splitting type} of $E$. When $E$ is a rank~$2$ bundle with first Chern class $c_1 (E) = j_1 + j_2 = 0$, we set $j = \vert j_1 \vert = \vert j_2 \vert$ so that $j \geq 0$ and say that $E$ is of {\it splitting type} $j$.
\end{definition}

Expressing Theorem~\ref{filtrable} in canonical coordinates gives that a rank~$2$ bundle $E$ with first Chern class $c_1 (E) = 0$ and splitting type $j$ can be defined via a transition matrix
\cite[Thm.\ 3.3]{gasparim1}
\begin{equation*}
\begin{pmatrix}
z^j &   p    \\
 0  & z^{-j}
\end{pmatrix}
=
\begin{pmatrix}
z^j & z^j \sigma \\
 0  &   z^{-j}
\end{pmatrix}
\end{equation*}
where $\sigma$ has the general form
\begin{align*}
\sigma = \sum_{i=1}^{\left\lfloor\!\frac{2j - 2}{k}\!\right\rfloor} \sum_{l = ik-2j+1}^{-1} \sigma_{il} z^l u^i, \qquad \sigma_{il} \in \mathbb C
\end{align*}
and in contrast to (\ref{generalcocycle}) the first sum now runs from $i = 1$, so that $\sigma \vert_\ell = 0$ and $E \vert_\ell = \mathcal O (j) \oplus \mathcal O (-j)$.

\subsection{Moduli} Moduli spaces of rank~$2$ bundles on $Z_k$ were studied in 
\cite[Thm.\ 3.5]{gasparim2} for the case of $Z_1$ and in \cite[Thm.\ 4.11]{ballicogasparimkoppe2} 
for the cases of $k \geq 1$. One could either give an ad-hoc definition of stability \cite[Def.\ 5.2]{GKM} and obtain quasi-projective varieties corresponding to moduli spaces of {\it framed-stable} bundles, or else take the point of view of stacks and study the full moduli stack of bundles on $Z_k$. The latter approach was taken in \cite{benbassatgasparim} where two equivalent presentations of the stack of bundles on $Z_k$ were given \cite[Thm.\ 3.1]{benbassatgasparim}. The former approach of choosing a definition of stability allows us to describe moduli spaces of framed-stable bundles with local second Chern class $j$ on $Z_k$, which turn out to be smooth quasi-projective varieties of dimension ${2j-k-2}$ \cite[Thm.\ 4.11]{ballicogasparimkoppe2}.

\section{Classical deformations}
\label{classicaldeformations}

Classical deformations of complex structures on a complex manifold $X$ are parametrized by $\H^1 (X, \mathcal T_X)$, where $\mathcal T_X$ is the (holomorphic) tangent bundle, with obstructions to deformation lying in $\H^2 (X, \mathcal T_X)$.

Although the general existence results in the deformation theory of Kodaira and Spencer \cite{kodaira} only hold for compact manifolds, much of the theory still applies to families of noncompact manifolds. We construct a semiuniversal family for the surfaces $Z_k$ explicitly by considering deformations of the vector bundle structure on $Z_k$ to an affine bundle structure.

\subsection{Deformations as affine bundles}

The vector bundle structure of $Z_k = \Tot \mathcal O_{\mathbb P^1} (-k)$ may be deformed to the structure of a (holomorphic) affine bundle over $\mathbb P^1$.

\begin{theorem}
\label{affinebundle}
Any deformation of $Z_k$ as affine bundle over $\mathbb P^1$ is given (in canonical coordinates) by an affine bundle structure of the form
\[
u \mapstotikz z^k u + \textstyle\sum\limits_{i=1}^{k-1} t_i z^i.
\]
\end{theorem}

\begin{proof}
As a vector bundle over $\mathbb P^1$, the transition function of $\mathcal O_{\mathbb P^1} (-k)$ is given in coordinates by $u \mapstotikz z^k u$, where $u$ is the fibre coordinate. An affine bundle structure is given by $T \colon u \mapstotikz z^k u + t (z)$, where $t (z)$ is a holomorphic function on $\mathbb C^*$.

The general form of $t (z)$ is
\[
t (z) = \sum_{i=-\infty}^\infty t_i z^i, \qquad b_i \in \mathbb C.
\]
An isomorphism of affine bundles $T \totikz A_V T A_U$ reducing $t (z)$ to the form $\sum_{i=1}^{k-1} t_i z^i$ is given in charts by
\begin{alignat*}{3}
A_U \colon u &\mapstotikz u - \sum_{i=0}^\infty t_{k+i} z^i  &&\qquad \text{on $U$} \\
A_V \colon v &\mapstotikz v - \sum_{i=0}^\infty t_{-i} \xi^i    &&\qquad \text{on $V$.}
\end{alignat*}
Clearly the choice of affine transformations on $U$ and $V$ are holomorphic with holomorphic inverse in the respective coordinates.
\end{proof}

\subsection{Deformations of the complex structure}

In this section we write the deformations of Theorem \ref{affinebundle} as a family of complex manifolds. To relate the deformations as affine bundle to deformations of complex structure of the total space, we first describe the (holomorphic) tangent bundle of $Z_k$ and calculate its first cohomology group $\H^1 (Z_k, \mathcal T_{Z_k})$.

In canonical coordinates the transition matrix for the tangent bundle of $Z_k$ is given by the Jacobian matrix
\begin{align}
\label{jacobian}
J =
\begin{pmatrix}
\tfrac{\partial}{\partial z} \, z^{-1} & \tfrac{\partial}{\partial u} \, z^{-1} \\[.75ex]
\tfrac{\partial}{\partial z} \, z^k u  & \tfrac{\partial}{\partial u} \, z^k u
\end{pmatrix}
=
\begin{pmatrix}
  -z^{-2}   &  0  \\
k z^{k-1} u & z^k
\end{pmatrix}
\end{align}
which shows that $\mathcal T_{Z_k}$ fits into a short exact sequence
\begin{equation*}
0 \rightarrowtikz \mathcal O (-k) \rightarrowtikz \mathcal T_{Z_k} \rightarrowtikz \mathcal O (2) \rightarrowtikz 0 \text{.}
\end{equation*}

\begin{lemma}
\label{H1tangent}
\begin{enumerate}
\item $\H^1 (Z_1, \mathcal T_{Z_1}) = 0$.
\item Let $k \geq 2$. Then $\H^1 (Z_k, \mathcal T_{Z_k}) \simeq \mathbb C^{k-1}$ and in $U$-coordinates a cohomology class $\tau \in \H^1 (Z_k, \mathcal T_{Z_k})$ is represented by a $1$-cocycle of the general form
\[
\sum_{i=1}^{k-1} t_i z^{-k+i} \frac{\partial}{\partial u}.
\]
\end{enumerate}
\end{lemma}

\begin{proof}
Let $\sigma \in \mathcal T_{Z_k} (U \cap V)$ be a general $1$-cocycle, written in $U$-coordinates. In the basis $\big\{ \frac{\partial}{\partial z}, \frac{\partial}{\partial u} \big\}$, $\sigma$ may be written as a convergent power series
\begin{align*}
\sigma = \sum_{i=0}^\infty \sum_{j=-\infty}^\infty
\begin{pmatrix}
a_{ij} \\
b_{ij} 
\end{pmatrix} z^j u^i,
\qquad a_{ij}, b_{ij} \in \mathbb C.
\end{align*}
Since terms with positive powers of $z$ are holomorphic on $U$ we obtain the cohomological equivalence
\[
\sigma \sim \sum_{i=0}^\infty \sum_{j=-\infty}^{-1}
\begin{pmatrix}
a_{ij} \\
b_{ij} 
\end{pmatrix} z^j u^i \text{.}
\]
Changing charts we have
\begin{align*}
J \sigma &\sim 
\begin{pmatrix}
     -z^{-2}   &  0  \\
k \, z^{k-1} u & z^k
\end{pmatrix}
\sum_{i=0}^\infty \sum_{j=-\infty}^{-1}
\begin{pmatrix}
a_{ij} \\
b_{ij}
\end{pmatrix}
z^j u^i \\
&= \sum_{i=0}^\infty \sum_{j=-\infty}^{-1}
\begin{pmatrix}
-z^{-2} \, a_{ij} \\
k \, z^{k-1} u \, a_{ij} + z^k \, b_{ij} 
\end{pmatrix}
z^j u^i
\end{align*}
and since monomials of the form $z^m u^n$ with $m \leq nk$ are holomorphic on $V$, we obtain
the cohomological equivalence
\[
J \sigma \sim \sum_{j=-k+1}^{-1}
\begin{pmatrix}  
0 \\
b_{0j}
\end{pmatrix}
z^{k+j}.
\]
On the $U$ chart, the nontrivial terms appearing in the expression of $\sigma$ are thus $z^{-k+i} \frac{\partial}{\partial u}$ for $1 \leq i \leq k-1$.
\end{proof}

The deformations constructed in Theorem \ref{affinebundle} may be given as a family of noncompact manifolds.

\begin{theorem}\label{family}
Let $k \geq 2$. Then $Z_k$ admits a $(k{-}1)$-dimensional semiuniversal family $Z_k \totikz M \toarg{\pi} \mathbb C^{k-1} \simeq \H^1 (Z_k, \mathcal T_{Z_k})$ of deformations.
\end{theorem}
\begin{proof}
Let $B = \mathbb C^{k-1}$ with coordinates $t_1, \dotsc, t_{k-1}$ and consider the complex manifold $M$ given by the charts
\[
U \times B = \{(z, u, t_1, \dotsc, t_{k-1})\} \qquad V \times B = \{(\xi, v, t_1, \dotsc, t_{k-1})\},
\]
with transition matrix
\[
\begin{pmatrix}
       z^{-2}    &  0  & 0 \\
z^{-1} \big( \textstyle\sum_{i=1}^{k-1} t_i z^i \big) & z^k & 0 \\
         0       &  0  & I_{k-1}
\end{pmatrix} \text{.}
\]
Then the projection $\pi \colon M \totikz B$ defines a family of noncompact manifolds with $M_0 = \pi^{-1} (0) \simeq Z_k$ and the fibre $M_t = \pi^{-1} (t)$ for $t = (t_1, \dotsc, t_{k-1}) \in \mathbb C^{k-1}$ is isomorphic to the total space of the affine bundle with affine structure $u \mapstotikz z^k + \sum_{i=1}^{k-1} t_i z^i$.

Recall from \cite[Def.\ 1.35]{manetti} that a family is {\it semiuniversal}, if it is {\it versal} (or {\it complete}) and the Kodaira--Spencer map $\operatorname{KS} \colon T_0 B \totikz \H^1 (Z_k, \mathcal T_{Z_k})$ is an isomorphism. The Kodaira--Spencer map of the family $\pi \colon M \totikz B$ is the vector space isomorphism
\begin{align}
\label{kodairaspencer}
\begin{tikzpicture}[baseline=-2.6pt,description/.style={fill=white,inner sep=2pt}]
\matrix (m) [matrix of math nodes, row sep=.7em, text height=1.5ex, column sep=0em, text depth=0.25ex, ampersand replacement=\&, column 3/.style={anchor=base west}, column 1/.style={anchor=base east}]
{\operatorname{KS} \colon \&[-.8em] T_0 B \&[2em] \H^1 (Z_k, \mathcal T_{Z_k}) \\
\& \displaystyle\frac{\partial}{\partial t_i} \& \displaystyle z^{-k+i} \frac{\partial}{\partial u} \text{.} \\};
\path[->,line width=.45pt,font=\scriptsize]
(m-1-2) edge (m-1-3)
;
\path[|->,line width=.45pt]
(m-2-2) edge (m-2-3)
;
\end{tikzpicture}
\end{align}
Identifying $B = \mathbb C^{k-1}$ with its tangent space at $0$ and using (\ref{kodairaspencer}) we denote a fibre $M_t$ for $t = (t_1, \dotsc, t_{k-1}) \in \mathbb C^{k-1}$ by $Z_k (\tau)$ where $\tau \in \H^1 (Z_k, \mathcal T_{Z_k})$ is the class represented by the cocycle $ \sum_{i=1}^{k-1} t_i z^{-k+i} \frac{\partial}{\partial u}$. Theorem \ref{affine} shows that $Z_k (\tau)$ is an affine algebraic variety for each $\tau \neq 0$, whence $\H^1 (Z_k (\tau), \mathcal T_{Z_k (\tau)}) = 0$. Thus $Z_k (\tau)$ admits no infinitesimal deformations. For each fibre $M_t$ the family $Z_k \totikz M \toarg{\pi} \mathbb C^{k-1}$ thus trivially contains all infinitesimal deformations of $M_t$ and is therefore versal.
\end{proof}

\begin{notation}
We fix {\it canonical coordinates} for the noncompact surfaces $Z_k (\tau)$, where
$\tau \in \H^1 (Z_k, \mathcal T_{Z_k})$ is the cohomology class of the $1$-cocycle $\sum_{i=1}^{k-1} t_i z^{-k+i} \frac{\partial}{\partial u}$. Let $Z_k (\tau) = U \cup V $ with coordinates $U = \{ (z,u) \}$, $V = \{ (\xi, v) \}$, 
such that on $U \cap V \simeq \mathbb C^* \times \mathbb C$ we identify
\begin{equation}
\label{glue}
\boxed{(\xi, v) = \Big( z^{-1}, z^k u + \textstyle\sum\limits_{i=1}^{k-1} t_i z^i \Big)\text{.}}
\end{equation}
$Z_k (\tau)$ is the total space of the affine line bundle given by the transition function $u \mapstotikz z^k u + \sum_{i=1}^{k-1} t_i z^i$.
\end{notation}

In this paper we only consider deformations of $Z_k$ of the form $Z_k (\tau)$. For $\tau \neq 0$, the complex structure of $Z_k (\tau)$ is different from the complex structure on $Z_k$ as we will see when we show that $Z_k (\tau)$ contains no compact complex analytic curves (Thm.\ \ref{nocompactcurves}), that every holomorphic vector bundle on $Z_k (\tau)$ splits as a direct sum of line bundles (Thm.\ \ref{decomposable}), and that $Z_k (\tau)$ admits the structure of a (smooth) affine complex algebraic variety (Thm.\ \ref{affine}).

\begin{remark}
Another way of arriving at the family of Theorem \ref{family} is to consider deformations of the tangent bundle $\mathcal T_{Z_k}$ as a vector bundle over $Z_k$ and ask which deformations could be the tangent bundle of a different complex manifold. Deformations of the tangent bundle that ``integrate'' to a different complex structure on $Z_k$ give rise precisely to the deformations $Z_k (\tau)$.
\end{remark}

\begin{remark}
For $k = 2$ the family given in Theorem~\ref{family} is already well known: it is the simultaneous resolution of the $A_1$ surface singularity (the rational double point) of Atiyah \cite{atiyah}. The deformations of $Z_2$ may also be obtained via the methods of Kaledin--Verbitsky \cite{kaledinverbitsky} who present a Torelli-type theorem for noncompact manifolds with a holomorphic symplectic form. However, of the surfaces $Z_k$, only $Z_2 \simeq T^* \mathbb P^1$ admits a holomorphic symplectic form.
\end{remark}

\subsection{Relation to deformations of Hirzebruch surfaces}
\label{relationtohirzebruch}

Recall from \cite{manetti} that the Hirzebruch surface $F_k$ is isomorphic to the subvariety of $\mathbb P^1 \times \mathbb P^{k+1}$ given as the set of points $([z_0 : z_1], [x_0 : \dotsc : x_{k+1}])$ satisfying the equation
\[
z_0 (x_1, \dotsc, x_k) = z_1 (x_2, \dotsc, x_{k+1}).
\]
The embedding $Z_k \hookrightarrowtikz F_k$ may be given by
\begin{equation}
\begin{aligned}
\label{embedding}
  (z, u) &\mapstotikz (  [1 : z], [1 : \eqmakebox[emb1]{$z^k u$} : \eqmakebox[emb2]{$z^{k-1} u$} : \ldots : \eqmakebox[emb3]{$u$}]) \\
(\xi, v) &\mapstotikz ([\xi : 1], [1 : \eqmakebox[emb1]{$v$}     : \eqmakebox[emb2]{$\xi \, v$}  : \ldots : \eqmakebox[emb3]{$\xi^k v$}])\text{.}
\end{aligned}
\end{equation}

Since $\H^2 (F_k, \mathcal T_{F_k}) = 0$, the Theorem of Existence \cite[Thm.\ 5.6]{kodaira} gives a semiuniversal family over a $(k{-}1)$-dimensional base, where $\dim \H^1 (F_k, \mathcal T_{F_k}) = k-1$. This family is given explicitly in \cite[\S 2.3]{manetti} as follows.

The family is given as $F_k \totikz \wtilde M \toarg{\tilde\pi} \mathbb C^{k-1}$, where $\wtilde M \subset \mathbb P^1 \times \mathbb P^{k+1} \times \mathbb C^{k-1}$ is the set of coordinates $([z_0 : z_1], [x_0 : \ldots : x_{k+1}], (t_1, \dotsc, t_{k-1}))$ satisfying the equation
\begin{align}
\label{familyhirzebruch}
z_0 (x_1, \dotsc, x_k) = z_1 (x_2 + t_1 x_0, \dotsc, x_k + t_{k-1} x_0, x_{k+1})
\end{align}
and $\tilde \pi = \pi_3 \vert_{\widetilde M}$ is the restriction of the projection to the third factor.

\begin{proposition}
Let $Z_k \totikz M \totikz \mathbb C^{k-1}$ be the family of Theorem \ref{family} and let $F_k \totikz \wtilde M \totikz \mathbb C^{k-1}$ be the family (\ref{familyhirzebruch}).

There is a commutative diagram
\begin{align*}
\begin{tikzpicture}
\matrix (m) [matrix of math nodes, inner sep=3.5pt, row sep=2.5em, text height=1.7ex, column sep=3em, text depth=0.25ex, ampersand replacement=\&, column 3/.style={anchor=base west}]
{
Z_k \&            M \& \mathbb C^{k-1}\phantom{.} \\
F_k \& \wtilde M \& \mathbb C^{k-1}. \\
};
\path[right hook->,line width=.45pt,font=\scriptsize]
(m-1-1) edge (m-2-1)
(m-1-2) edge (m-2-2)
;
\path[->,line width=.45pt,font=\scriptsize]
(m-1-1) edge (m-1-2)
(m-2-1) edge (m-2-2)
(m-1-2) edge (m-1-3)
(m-2-2) edge (m-2-3)
;
\draw [double equal sign distance,line width=.45pt]
(m-1-3) to (m-2-3)
;
\end{tikzpicture}
\end{align*}
\end{proposition}

\begin{proof}
Recall from the proof of Theorem \ref{family} that the family $M$ is covered by the sets $U \times B$ and $V \times B$, where $U, V$ are the canonical charts and $B \simeq \mathbb C^{k-1}$, such that on $(U \times B) \cap (V \times B)$ we identify
\begin{align}
\label{identify}
(\xi, v, t_1, \dotsc, t_{k-1}) = \Big( z^{-1}, z^k u + \textstyle\sum\limits_{i=1}^{k-1} t_i z^i, t_1, \dotsc, t_{k-1} \Big).
\end{align}
Define the map $M \hookrightarrowtikz \wtilde M$ on charts $U \times B$ and $V \times B$ by
\begin{equation}
\begin{aligned}
\label{embeddingfamily}
(z,   u, t_1, \dotsc, t_{k-1}) &\mapstotikz ([ 1  : z], [1 : \eqmakebox[embf1]{$x_1$} : \ldots : \eqmakebox[embf2]{$x_{k+1}$}], t_1, \dotsc, t_{k-1}) \\
(\xi, v, t_1, \dotsc, t_{k-1}) &\mapstotikz ([\xi : 1], [1 : \eqmakebox[embf1]{$y_1$} : \ldots : \eqmakebox[embf2]{$y_{k+1}$}], t_1, \dotsc, t_{k-1}),
\end{aligned}
\end{equation}
where
\begin{alignat*}{5}
x_{n+1} &={} &z^{k-n} u + \eqmakebox[sum]{$\textstyle\sum\limits_{i=n+1}^{k-1}$} t_i &z^{i-n} \\
y_{n+1} &={} &\xi^n   v - \eqmakebox[sum]{$\textstyle\sum\limits_{i=1}^n$}       t_i &\xi^{n-i}
\end{alignat*}
for $0 \leq n \leq k$. (The sum is empty for $x_k$, $x_{k+1}$ and $y_1$.)

One checks that the map (\ref{embeddingfamily}) is injective and satisfies (\ref{familyhirzebruch}) on each chart; moreover, the map is well defined on the intersection, which follows from using the identity (\ref{identify}).

The commutativity of the first square follows from (\ref{embedding}) for $t_1 = \dotsb = t_{k-1} = 0$. The commutativity of the second square is evident from (\ref{embeddingfamily}).
\end{proof}

\begin{remark}[(An application to mirror symmetry)]
\label{mirrorsymmetry}
A recent result \cite{GGS} shows that adjoint orbits of semisimple Lie groups have the structure of symplectic Lefschetz fibrations. Such Lefschetz fibrations are considered in the homological mirror symmetry conjecture, where one needs to identify details of its complex and symplectic structures. The adjoint orbit of $\operatorname{\mathfrak{sl}} (2, \mathbb C)$ is isomorphic to a nontrivial deformation $Z_2 (\tau)$ of the noncompact surface $Z_2$. By Theorem \ref{family}, the deformation $Z_2 (\tau)$ is unique up to scaling. Using Lie theoretical methods, \cite{BBGGS} considered a Landau--Ginzburg model over $Z_2 (\tau)$, proved it does not admit any projective mirror, and identified its mirror category as a proper subcategory of coherent sheaves on the second Hirzebruch surface. For $k \neq 2$ the role played by the deformations $Z_k (\tau)$ in mirror symmetry is not yet known. However, we expect deformations of other noncompact manifolds to play an important role in mirror symmetry.
\end{remark}

\section{Geometry and topology of $Z_k (\tau)$}
\label{geometrydeformations}

\subsection{Line bundles on $Z_k (\tau)$}

\begin{lemma}
$\H^1 (Z_k (\tau), \mathcal O) = \H^2 (Z_k (\tau), \mathcal O) = 0$.
\end{lemma}
\begin{proof}
A general $1$-cocycle can be written in the form
\[
\alpha = \sum_{i=0}^\infty \sum_{l=-\infty}^\infty \alpha_{il} z^l u^i \sim
 \sum_{i=0}^\infty \sum_{l=-\infty}^{-1} \alpha_{il} z^l u^i.
\]
In canonical coordinates, we have $u = \xi^k v - \sum_{n=1}^{k-1} t_n \xi^{k-n}$.
Thus, on the $V$-chart 
\[
\alpha \sim \sum_{i=0}^\infty \sum_{l=1}^\infty \alpha_{i,-l} \xi^l \Big( \xi^k v - \textstyle\sum\limits_{n=1}^{k-1} t_n \xi^{k-n} \Big)^i \sim 0
\]
since positive powers of $\xi, v$ are holomorphic on $V$.

$\H^2 (Z_k (\tau), \mathcal F) = 0$ for any coherent sheaf $\mathcal F$ of coefficients, since $Z_k (\tau)$ is Leray-covered by two open sets.
\end{proof}

We thus get the following isomorphisms,
\[
\Pic Z_k (\tau) \simeq \H^2 (Z_k (\tau), \mathbb Z) \simeq \H^2 (Z_k, \mathbb Z) \simeq \H^2 (S^2, \mathbb Z) \simeq \mathbb Z,
\]
where the first isomorphism follows from the exponential sheaf sequence for $Z_k (\tau)$ and the second isomorphism
follows from the fact that $Z_k (\tau)$ is homeomorphic to $Z_k$ as a real manifold. So, any line bundle on $Z_k (\tau)$ is determined by its first Chern class. We write $\mathcal O_{Z_k (\tau)} (n)$ or $\mathcal O (n)$ for the line bundle on $Z_k (\tau)$ with first Chern class $n$.

\begin{lemma}
\label{chernclass}
The line bundle on $Z_k (\tau)$ with first Chern class $n$, denoted $\mathcal O (n)$, can be given the transition matrix $\big(z^{-n}\big)$.
\end{lemma}

\begin{proof}[First proof]
Let $M$ be the total space of the family given in the proof of Theorem \ref{family}. The matrix $\big(z^{-n}\big)$ defines a line bundle $\mathcal L$ over $M$. The restriction of this line bundle to the central fibre of the family gives $\mathcal O_{Z_k} (n)$ (with transition matrix $\big(z^{-n}\big)$ and Chern class $n$). Since the family is continuous and the first Chern class is a discrete topological invariant it remains constant in the family, hence $c_1 (\mathcal L \vert_{Z_k (\tau)}) = n$ as well.
\end{proof}

\begin{proof}[Second proof]
Define a map
\begin{align}
\begin{tikzpicture}[baseline=-2.6pt,description/.style={fill=white,inner sep=2pt}]
\matrix (m) [matrix of math nodes, row sep=0em, text height=1.5ex, column sep=0em, text depth=0.25ex, ampersand replacement=\&, column 3/.style={anchor=base west}]
{ \phi \colon \&[-.75em] Z_k (\tau) \&[2em] \mathbb P^1 \\
              \&               (z, u)     \&       {[1 : z]}  \\
              \&             (\xi, v)     \&     {[\xi : 1 ]}.\\};
\path[->,line width=.45pt,font=\scriptsize]
(m-1-2) edge (m-1-3)
;
\path[|->,line width=.45pt,font=\scriptsize]
(m-2-2) edge (m-2-3)
(m-3-2) edge (m-3-3)
;
\end{tikzpicture}
\end{align}
If we denote by $T$ the transition function of $\mathcal O_{\mathbb P^1} (n)$, then the transition function of the pullback bundle $\phi^* \mathcal O_{\mathbb P^1} (n) \simeq \mathcal O_{Z_k (\tau)} (n)$ is $\phi^* T = T \circ \phi$, which in canonical coordinates is precisely multiplication by $\big(z^{-n}\big)$.
\end{proof}

\subsection{Cohomology of $Z_k (\tau)$}

We calculate sheaf cohomology of $Z_k (\tau)$ with coefficients in line bundles.

\begin{remark}
\label{H2}
As $Z_k (\tau)$ is covered by two open sets with acyclic intersection, we have that $\H^i (Z_k (\tau), \mathcal F) = 0$ for $i \geq 2$ and any coherent sheaf $\mathcal F$.
\end{remark}

\begin{lemma}
\label{h1deformed}
Let $Z_k (\tau)$ be any nontrivial deformation of $Z_k$. Then $\H^1 (Z_k (\tau), \mathcal O (-n)) = 0$ for any integer $n$.
\end{lemma}

\begin{proof}
As for $Z_k$, this is straightforward if $n \leq 1$. We thus assume that $n \geq 2$. The idea of the proof is to use the fact that a function holomorphic on $V \subset Z_1$ is also holomorphic on $V \subset Z_k (\tau)$ ({\it cf.}\ Lem.\ \ref{stillholomorphic}). The difficult part is thus to show that all cocycles which are {\it nontrivial} in $\H^1 (Z_1, \mathcal O (-n))$ are in fact {\it trivial} in $\H^1 (Z_k (\tau), \mathcal O (-n))$, because all other terms can be removed by functions on $Z_1$. These cocycles are spanned by the terms $z^l u^i$, where $-n + i < l < 0$ and $i \leq n - 2$.

Let $\sigma = \sum_{i=0}^\infty \sum_{l=-\infty}^\infty \sigma_{il} z^l u^i$ be a general $1$-cocycle, {\it i.e.}\ a holomorphic function in the intersection $U \cap V \simeq \mathbb C^* \times \mathbb C$. We may add any function $z^{-n} f$, where $f$ is holomorphic on $V$ without changing the cohomology class of $\sigma$.

First we remove the terms $z^l u^i$ for $-n + i < l < 0$ and $i \leq n - 2$ as follows. Let $t_1, \dotsc, t_{k-1}$ be the constants in (\ref{glue}) and let $m$ be the smallest integer such that $t_m \neq 0$. We start with $i = 0$ and add a suitable multiple of $\xi^{n - 1 + n m} v^n = z^{-n+1} (t_m^n + O (z, u))$ to remove $\sigma_{0,-n+1} z^{-n+1}$. While removing the coefficient of $z^{-n+1}$, we only add to the coefficients of higher powers of $z$ that will be removed in subsequent steps. We continue in the same fashion for the terms $z^l$, where $-n + 1 < l < 0$ by adding suitable multiples of $\xi^{n - s + n m} v^n$ for $1 < s < n$.

For $i \geq 1$, note that $u = \xi^k v - \sum_{i=1}^{k-1} t_i \xi^{k-i}$ is holomorphic on $V$, so that we may use the expressions
\[
\Big( \xi^k v - \textstyle\sum\limits_{n=1}^{k-1} t_n \xi^{k-n} \Big)^i \xi^{n - s + n m} v^n = z^{-n+s} u^i (1 + O (z, u)),
\]
where $1 \leq s \leq n$, to remove the remaining terms $z^l u^i$, where $-n + i < l < 0$ and $i \leq n - 2$. We have added finitely many functions to $\sigma$ and now $\sigma \sim \sigma' = \sum_{i=0}^\infty \sum_{l=-\infty}^\infty \sigma'_{il} z^l u^i$, where the coefficients $\sigma'_{il} = 0$ for $-n + i < l < 0$ and $i \leq n - 2$.

Next, we may add any function $z^{-n} f_V = z^{-n} \sum_{i=0}^\infty \sum_{l=0} f^V_{il} z^{-l} (z u)^i$ to remove all remaining nonzero coefficients of $z^l u^i$ with $l \leq -n + i$ since $z u = \xi^{k-1} v - \sum_{n=1}^{k-1} t_n \xi^{k-n-1}$ is holomorphic on $V$. Finally, nonzero coefficients of the terms $z^l u^i$ for $i,l \geq 0$ may be removed by adding a suitable function holomorphic on $U$.
\end{proof}

\subsection{Subvarieties}

Deforming the complex structure does not change the topology of the manifold, thus for any $\tau$ we still have 
\[
\H^i (Z_k (\tau), \mathbb C)
=
\begin{cases}
\mathbb C & \text{if } i = 0, 2 \\
    0     & \text{otherwise.}
\end{cases}
\]
The de Rham cohomology of $Z_k$ comes from the complex submanifold $\ell \simeq \mathbb P^1 \approx S^2$. However, for a nontrivial deformation $Z_k (\tau)$, this is no longer the case, as the following result shows.

\begin{theorem}
\label{nocompactcurves}
Let $k \geq 2$. A nontrivial deformation of $Z_k$ contains no complex analytic compact curves.
\end{theorem}
\begin{proof}
Let $D$ be a $1$-dimensional subvariety of $Z_k (\tau)$ and let $\mathcal O (D)$ denote the associated line bundle. By Lemma~\ref{chernclass}, $\mathcal O (D)$ is isomorphic to $\mathcal O (j)$ for some $j$  and hence  $D$ can be obtained as the zero locus of a global holomorphic section of $\mathcal O (j)$. 

Recall that, in canonical coordinates, $\mathcal O (j)$ can be given by transition matrix $(z^{-j})$. We now construct a global section.
Over $U$, a global holomorphic section $s$ of $\mathcal O (j)$ is of the form
\[
s \vert_U = \sum_{i=0}^\infty \sum_{l=0}^\infty s_{il} z^l u^i.
\]
Expression (\ref{glue}) gives that $u = \xi^k v - \sum_{n=1}^{k-1} t_n \xi^{k-n}$ on $U \cap V$. Let $1 \leq m \leq k - 1$ be the largest integer such that $t_m \neq 0$. Changing coordinates with the transition function $z^{-j}$, we can write the zero locus of $s$ over $V$ as
\begin{equation}
\label{zerolocus}
s \vert_V = \sum_{i=0}^\infty \sum_{l=0}^\infty s_{il} \xi^{i(k-m) + j - l} \Big( \xi^m v - \textstyle\sum\limits_{n=1}^{m-1} t_n \xi^{k-n} - t_m \Big)^i = 0.
\end{equation}
The term in parentheses expands to
\[
\Big( \xi^m v - \textstyle\sum\limits_{n=1}^{m-1} t_n \xi^{k-n} - t_m \Big)^i = \xi^{im} v^i + \dotsb + (-1)^{i-1} i \, \xi^m v \Big( \sum\limits_{n=1}^{m-1} t_n \xi^{k-n} + t_m \Big)^{i-1} + (-1)^i t_m^i.
\]
Since $t_m \neq 0$ the coefficient $s_{il}$ must be zero if
\[
i (k-m) + j - l < 0
\]
as otherwise the sum would contain a term not holomorphic on $V$. But then each of the terms containing $v$ also contains a positive power of $\xi$.

Thus, on the fibre over $\xi = 0$, equation (\ref{zerolocus}) reduces to
\begin{equation}
\label{zerolocusfibre}
\sum_{i=0}^\infty (-1)^i s_{i,i(k-m)+j} \, t_m^i = 0
\end{equation}
and does not depend on $v$. If (\ref{zerolocusfibre}) holds, then $D$ contains the whole fibre over $\xi = 0$ as a component and is thus not compact; if (\ref{zerolocusfibre}) does not hold then $D$ is contained in $U \simeq \mathbb C^2$ which is affine and thus $D$ is also not compact.
\end{proof}

In particular, $Z_k (\tau)$ contains no complex submanifold with the topology of a 2-sphere for $\tau \neq 0$.

\subsection{Vector bundles on $Z_k (\tau)$}

We generalize the algebraicity and filtrability result for $Z_k$ (Thm.\ \ref{filtrable}) to its classical deformations. We start with a technical lemma.

\begin{lemma}
\label{stillholomorphic}
Let $f_V$ be a function holomorphic on $V \subset Z_1$. Then there is a holomorphic function $f_{\tilde V}$ holomorphic on $V \subset Z_k (\tau)$ such that, written in $U$-coordinates, we have ${f_V}\vert_{U\cap V} = {f_{\tilde V}}\vert_{U\cap \tilde V}$.
\end{lemma}

\begin{proof}
To differentiate the coordinates of $V \subset Z_1$ and $V \subset Z_k (\tau)$, write $\wtilde V = \{ (\tilde \xi, \tilde v) \}$ for the chart of $Z_k (\tau)$, {\it i.e.}\
\begin{alignat*}{4}
       (\xi,        v) &= (z^{-1}, z u)          &&\qquad \text{on $U \cap V \subset Z_1$} \\
(\tilde \xi, \tilde v) &= (z^{-1}, z^k u + \tau) &&\qquad \text{on $U \cap \wtilde V \subset Z_k (\tau)$.}
\end{alignat*}
Now a holomorphic function $f_V$ on $V$ may be written as a convergent power series in the variables $z^{-1}$ and $z u$,
\[
\sum_{i=0}^\infty \sum_{l=0}^\infty f_{il} z^{-l} (z u)^i \text{.}
\]
Now on $\wtilde V \subset Z_k (\tau)$, we have that $z^{-1} = \tilde \xi$ and $z u = \tilde \xi^{k-1} \tilde v - \sum_{i=1}^{k-1} t_i \tilde \xi^{k-i-1}$. We may rewrite $f_V$ in $\wtilde V$-coordinates
\begin{align}
\label{beforeexpansion}
f_{\tilde V} = \sum_{i=0}^\infty \sum_{l=0}^\infty f_{il} \tilde \xi^l \Big( \tilde \xi^{k-1} \tilde v - \textstyle\sum\limits_{i=1}^{k-1} t_i \tilde \xi^{k-i-1} \Big)^i \text{.}
\end{align}
This is a convergent power series in the variables $\tilde \xi$ and $\tilde v$ which we may write as
\[
f_{\tilde V} = \sum_{i=0}^\infty \sum_{l=0}^\infty \tilde f_{il} \tilde \xi^l \tilde v^i
\]
by expanding the factor $\big( \tilde \xi^{k-1} \tilde v - \sum_{i=1}^{k-1} t_i \tilde \xi^{k-i-1} \big)^i$ in (\ref{beforeexpansion}) and rewriting the coefficients. For example, $\tilde f_{00} = \sum_{i=0}^\infty f_{i0} t_{k-1}^i$. Thus $f_{\tilde V}$ is holomorphic on $\wtilde V \subset Z_k (\tau)$. Since both $f_V$ and $f_{\tilde V}$ come from the same power series in $(z, u)$-coordinates, we have $f_V = f_{\tilde V}$ in $(z, u)$-coordinates.
\end{proof}

\begin{theorem}
\label{filtrabledeformed}
Holomorphic bundles over $Z_k (\tau)$ are algebraic and filtrable.
\end{theorem}
\begin{proof}
As in Lemma \ref{stillholomorphic}, denote by $\wtilde V$ the chart of $Z_k (\tau)$ with coordinates $\tilde \xi$ and $\tilde v$.

Let $\wtilde E$ be a vector bundle of rank~$r$ over $Z_k (\tau)$ with transition function $T$. The entries of $T$ are functions holomorphic in the intersection $U \cap \wtilde V$ and may thus be written as power series
\[
\sum_{i=0}^\infty \sum_{l=-\infty}^\infty a_{il} z^l u^i.
\]
The same $T$ also defines a bundle $E$ on $Z_1$. As a bundle over $Z_1$, $E$ is algebraic and filtrable by Theorem \ref{filtrable}. In particular, there exist matrices $A_U, A_V$ with entries holomorphic in $U, V \subset Z_1$, respectively, which are invertible for all points in $U, V$, respectively, and such that
\[
A_V T A_U = 
\left(
\begin{tikzpicture}[x=1.8em,y=1.8em,baseline=-4.6em]
\node[inner sep=0pt,shape=circle,baseline=(11.base)] (11) at (1,-1) {$z^{j_1}$\strut};
\node[inner sep=0pt,shape=circle,baseline=(44.base)] (44) at (4,-4) {$z^{j_r}$\strut};
\path[dash pattern=on 0pt off 5pt, line width=1.2pt, line cap=round] (11) edge (44);
\node[inner sep=0pt,shape=circle] (12) at (2,-1) {$\ast$\strut};
\node[inner sep=0pt,shape=circle] (14) at (4,-1) {$\ast$\strut};
\node[inner sep=0pt,shape=circle] (34) at (4,-3) {$\ast$\strut};
\path[dash pattern=on 0pt off 5pt, line width=1.2pt, line cap=round] (12) edge (14);
\path[dash pattern=on 0pt off 6pt, line width=1.2pt, line cap=round] (12) edge (34);
\path[dash pattern=on 0pt off 5pt, line width=1.2pt, line cap=round] (14) edge (34);
\node[inner sep=0pt,shape=circle,baseline=(11.base)] (21) at (1,-2) {$0$\strut};
\node[inner sep=0pt,shape=circle,baseline=(43.base)] (41) at (1,-4) {$0$\strut};
\node[inner sep=0pt,shape=circle,baseline=(43.base)] (43) at (3,-4) {$0$\strut};
\path[dash pattern=on 0pt off 5pt, line width=1.2pt, line cap=round] (21) edge (41);
\path[dash pattern=on 0pt off 5pt, line width=1.2pt, line cap=round] (41) edge (43);
\path[dash pattern=on 0pt off 6pt, line width=1.2pt, line cap=round] (21) edge (43);
\end{tikzpicture}
\right)
\]
where each $\ast$ denotes an {\it algebraic} function on $U \cap V$. By Lemma \ref{stillholomorphic}, we conclude that the entries of $A_U, A_V$ are also holomorphic in $U, \wtilde V \subset Z_k (\tau)$ and thus $A_U, A_V$ also define an isomorphism of $\wtilde E$ with a filtered algebraic bundle.
\end{proof}

\begin{remark}
For the proof of Theorem \ref{filtrabledeformed} we do not need to require $\tau \neq 0$. In other words, to prove algebraicity and filtrability for $Z_k$ (or any of its deformations), it is enough to prove it for $Z_1$.
\end{remark}

\begin{remark}
Note that on $\mathbb P^2$ the only bundles which are filtrable are split, but there is a multitude of bundles which are not filtrable, given that the moduli spaces of stable vector bundles on $\mathbb P^2$ are quasi-projective varieties whose dimensions increase with the second Chern classes, see \cite{okonekschneiderspindler}.
\end{remark}

Theorem \ref{filtrabledeformed} implies that, as for $Z_k$, rank~$2$ bundles over $Z_k (\tau)$ with vanishing first Chern class may be written as extensions of line bundles
\[
\Ext^1 (\mathcal O_{Z_k (\tau)} (j), \mathcal O_{Z_k (\tau)} (-j)) \simeq \H^1 (Z_k (\tau), \mathcal O (-2j))\text{,}
\]
but by Lemma \ref{h1deformed}, $\H^1 (Z_k (\tau), \mathcal O (-2j)) = 0$. Hence, we obtain:

\begin{theorem}
\label{decomposable}
Let $Z_k (\tau)$ be any nontrivial deformation of $Z_k$.
Then, every holomorphic vector bundle on $Z_k (\tau)$ splits as a direct sum of line bundles. 
\end{theorem}

\begin{remark} In contrast with $Z_k (\tau)$, the surfaces $Z_k$ have nontrivial moduli of vector bundles. For example, \cite[Thm.\ 4.11]{ballicogasparimkoppe2} shows that the moduli of rank~$2$ bundles on $Z_k$ with splitting type $j$ has dimension $2j-k-2$.
\end{remark}

\begin{remark} Consider the threefold $W_1 := \Tot (\mathcal O_{\mathbb P^1}(-1) \oplus \mathcal O_{\mathbb P^1}(-1))$.
In \cite[Thm.\ 4.3] {amilburubarmeiercallandergasparim} it is proven that, for appropriate choices 
of numerical invariants, there are isomorphisms of moduli spaces of vector bundles on $W_1$ and moduli spaces of vector bundles on the surfaces $Z_k$. Thus, from the point of view of moduli of bundles this threefold presents similar behaviour to our surfaces. However, from the point of view of deformation theory they are quite different. Indeed, we observe that $\H^1 (W_1, \mathcal T_{W_1}) = 0$, hence there are no classical deformations of $W_1$.
\end{remark}

\subsection{The affine structure of $Z_k (\tau)$}

As a corollary to the splitting principle of Theorem \ref{decomposable}, we are able to show that any nontrivial deformation $Z_k (\tau)$ of $Z_k$ admits the structure of an affine variety. We state part of the proof as a separate lemma.

\begin{lemma}
\label{resolutionproperty}
$Z_k (\tau)$ has the resolution property, {\it i.e.}\ every coherent sheaf has a global resolution by locally free sheaves.
\end{lemma}

\begin{proof}
This follows from \cite[Prop.\ 8.1]{totaro} by noting that $Z_k (\tau)$ has affine diagonal since intersections of affine sets are affine.
\end{proof}

\begin{theorem}
\label{affine}
Let $Z_k (\tau)$ be any nontrivial deformation of $Z_k$. Then $Z_k (\tau)$ admits the structure of a smooth affine algebraic variety.
\end{theorem}

\begin{proof}
Let $\mathcal F$ be an arbitrary coherent sheaf on $Z_k (\tau)$. By Lemma \ref{resolutionproperty} $\mathcal F$ admits a global resolution by locally free sheaves
\[
\dotsb \totikz \mathcal F_n \totikz \dotsb \totikz \mathcal F_0 \totikz \mathcal F \totikz 0.
\]
Since $Z_k (\tau)$ is of dimension two, we claim that the kernel of $\mathcal F_2 \toarg{d_1} \mathcal F_1$ is locally free. Since locally free is a local property, it suffices to check this property on the stalks $(\ker d_1)_x$ for $x \in Z_k (\tau)$. Since the local rings $\mathcal O_{Z_k (\tau), x}$ are regular and of dimension $\leq 2$, they are of global dimension $\leq 2$ and free. We may thus truncate the resolution to
\[
0 \totikz \ker d_1 \toarg{i} \mathcal F_1 \toarg{d_0} \mathcal F_0 \toarg{\epsilon} \mathcal F \totikz 0.
\]
We have short exact sequences
\begin{align*}
\begin{tikzpicture}[baseline=-2.3pt,description/.style={fill=white,inner sep=2pt}]
\matrix (m) [matrix of math nodes, row sep=0em,
column sep=2em, text height=1.5ex, text depth=0.25ex, ampersand replacement=\&, column 2/.style={anchor=base west}]
{
0 \& \ker d_1      \& \mathcal F_1 \& \ker d_0        \& 0 \\
0 \& \ker d_0      \& \mathcal F_1 \& \ker \epsilon   \& 0 \\
0 \& \ker \epsilon \& \mathcal F_0 \& \mathcal F      \& 0 \\
};
\draw[->,line width=.45pt,font=\scriptsize]
(m-1-1) edge (m-1-2)
(m-1-2) edge (m-1-3)
(m-1-3) edge (m-1-4)
(m-1-4) edge (m-1-5)
(m-2-1) edge (m-2-2)
(m-2-2) edge (m-2-3)
(m-2-3) edge (m-2-4)
(m-2-4) edge (m-2-5)
(m-3-1) edge (m-3-2)
(m-3-2) edge (m-3-3)
(m-3-3) edge (m-3-4)
(m-3-4) edge (m-3-5)
;
\end{tikzpicture}
\end{align*}
where we have used the isomorphisms $\coim i \simeq \im i$ and $\coim d_0 \simeq \im d_0$ which hold in any Abelian category, here $\Coh (Z_k)$, as well as $\im i \simeq \ker d_0$ and $\im d_0 \simeq \ker \epsilon$ since the resolution is an {\it exact} sequence.

We have that $\ker d_1$, $\mathcal F_1$ and $\mathcal F_0$ are acyclic, since they are direct sums of line bundles by Theorem \ref{decomposable}, which have no higher cohomology (by Lemma \ref{h1deformed}). Thus in each of the above sequences, the first two terms are acyclic, so is the third by applying the long exact sequence in cohomology iteratively. This proves that $\mathcal F$ is acyclic. Since $\mathcal F$ was arbitrary, applying Serre's criterion \cite[Thm.\ III.3.7]{hartshorne} we conclude that $Z_k (\tau)$ is affine.
\end{proof}

\begin{corollary}
\label{grauert}
Moduli spaces of holomorphic vector bundles with fixed topological invariants are trivial, {\it i.e.}\ consist of a single point.
\end{corollary}

\begin{proof}
Affine varieties are Stein and the Grauert--Oka principle \cite{grauert} implies that the holomorphic and topological classifications of complex vector bundles coincide.
\end{proof}

\begin{remark}
We used filtrability and algebraicity of vector bundles over the deformations together with the vanishing of their higher cohomology to prove affineness of the deformations.

The referee pointed out that, conversely, our results for vector bundles over the deformations may be deduced from the affineness of the deformations. Our proofs can thus be shortened substantially by first showing affineness of the deformations, which may be proved by appealing to the compactification to Hirzebruch surfaces whose deformation theory is well understood: When a Hirzebruch surface $F_k$ deforms to $F_m$ with $m < k$ (where $m \equiv k$ $\mathrm{mod}\;2$), the irreducible divisor $F_k \setminus Z_k$ with self-intersection $k$ decomposes into a sum of divisors, preserving the self-intersection. An irreducible divisor in $F_m$ has self-intersection at most $m$, and the divisor with self-intersection $k > m$ is the sum of the line at infinity with self-intersection $m$ as well as $\frac{k-m}2$ copies of the fibre. This sum is now an ample divisor as the intersection with any irreducible curve is positive; its complement, a deformation of $Z_k$, is therefore affine.

We nevertheless chose to maintain our direct proofs, working only with the noncompact surfaces, because this approach also proved to be useful for other noncompact spaces, which may admit larger families of deformations than those obtained from a compactification. For instance, for the Calabi--Yau threefold $W_2 := \Tot (\mathcal O_{\mathbb P^1} (-2) \oplus \mathcal O_{\mathbb P^1}) \simeq Z_2 \times \mathbb C$, the cohomology $\H^1 (W_2, \mathcal T_{W_2})$ is infinite dimensional over $\mathbb C$ and \cite{GRS,GS} show that indeed infinitely many directions of this vector space produce inequivalent deformations.
\end{remark}

\section{Applications to the theory of instantons}
\label{instantons}

$\mathrm{SU} (2)$-instantons on $Z_k$ correspond to framed rank~$2$ holomorphic bundles on $Z_k$ with vanishing first Chern class \cite[Prop.\ 5.3]{GKM}.

The charge of such an instanton is associated to the local holomorphic Euler characteristic of the bundle. We recall the definition.

\begin{definition}[\cite{GKM}]
Let $X_k$ denote the variety obtained from $Z_k$ by contracting the zero section to a point $x \in X_k$. Hence $X_k$ is singular at $x$ if $k > 1$. Let $\pi \colon Z_k \rightarrowtikz X_k$ denote the contraction map. Let $E$ be a rank~$2$ bundle (or any reflexive sheaf) on $Z_k$, and define a skyscraper sheaf $Q$ by the exact sequence
\[
0 \rightarrowtikz \pi_* E \rightarrowtikz (\pi_* E)^{\vee\vee} \rightarrowtikz Q \rightarrowtikz 0.
\]
Then the {\it local holomorphic Euler characteristic} of $E$
is 
\[
\label{euler}
\chi_{\mathrm{loc}} (E) := \chi(\pi_* E, x) = 
h^0 (X, Q) + 
h^0 (X, R^1\pi_*E) \text{.}
\]
In light of the Kobayashi--Hitchin correspondence for $Z_k$ proved in \cite[Prop.\ 5.3]{GKM}, we also refer to this number as the {\it charge} of a corresponding instanton on $Z_k$. 
\end{definition}

We observe that not all bundles on $Z_k$ correspond to instantons. In fact, \cite[Cor.\ 5.5]{GKM} shows that an $\operatorname{SL} (2, \mathbb{C})$-bundle $E$ over $Z_k$ represents an instanton if and only if its splitting type (see Def.\ \ref{splittingtype}) is a multiple of $k$. In this case \cite[Prop.\ 4.1]{GKM} implies that the restriction of $E$ to $Z_k \setminus \ell$ is trivial, and hence $E$ can be extended to a bundle $\overbar E$ on the Hirzebruch surface $F_k$ which is trivial along the line at infinity $\ell_\infty$.

As the restriction of $\overbar E$ to an analytic neighbourhood of $\ell_\infty$ is trivial, the charge of $\overbar E$ only depends on the neighbourhood of $\ell$, which is $Z_k$. That is, we have
\[
\chi (\overbar E) = \chi_{\mathrm{loc}} (E).
\]

The moduli space of instantons of charge $j \equiv 0$ $\mathrm{mod}\; k$ on $Z_k$ is a smooth quasi-projective variety of dimension $2j-k-2$ \cite[Thm.\ 4.11]{ballicogasparimkoppe2}. In terms of vector bundles the statement is that the moduli space of rank~$2$ holomorphic bundles on $Z_k$ with vanishing first Chern class and $\chi_{\mathrm{loc}} = j$ contains an open dense subset isomorphic to 
$\mathbb P^{2j-k-2}$ minus a closed subvariety of codimension at least $k + 1$.

Here, we do not develop a full theory of instantons on the deformed surfaces $Z_k (\tau)$. However, in light of \S \ref{relationtohirzebruch} we may apply our results to instantons by using the Kobayashi--Hitchin correspondence for Hirzebruch surfaces, and then restricting to $Z_k (\tau)$.

\begin{definition}
\label{extendstrivially}
Let $D$ be a divisor on a smooth complex surface $X$. We say that a holomorphic vector bundle $E$ on $X \setminus D$ {\it extends trivially} to $X$, if there exists a holomorphic vector bundle $\overbar E$ on $X$ such that $\overbar E \vert_{X \setminus D} = E$ and $\overbar E \vert_D$ is trivial.
\end{definition}

Now let $D$ be a divisor on the Hirzebruch surface $F_k$. The usual Kobayashi--Hitchin correspondence \cite{LT}
\[
\left\{
\begin{array}{c}
\text{irreducible $\mathrm{SU} (2)$-instantons} \\ 
\text{of charge $n$}
\end{array}
\right\}
\leftrightarrowtikz
\left\{
\begin{array}{c}
\text{stable $\operatorname{SL} (2, \mathbb C)$-bundles} \\
\text{with $c_2 = n$}
\end{array}
\right\}
\]
takes an instanton whose charge is held on $F_k \setminus D$ to a bundle which is trivial along $D$.

In further detail, we actually need to consider {\it framed} bundles on ${Z_k(\tau)}$ as follows. By \S \ref{relationtohirzebruch}, ${Z_k(\tau)} \subset F_m$ for some $m \leq k$. Let $D$ be the divisor $F_m \setminus {Z_k(\tau)}$, {\it i.e.}\ $F_m = {Z_k(\tau)} \sqcup D$ as a disjoint union. Let $N$ be a small open (analytic) tubular neighbourhood of $D$ in $F_m$, and denote by $N^\circ := N \setminus D$ the deleted neighbourhood. Thus, we can also think of $N^\circ$ as an open neighbourhood of infinity inside ${Z_k(\tau)}$; let us denote this open subset of ${Z_k(\tau)}$ by ${Z_k^\circ(\tau)}$.

\begin{definition}
\begin{enumerate}
\item A {\it framed} rank~$2$ bundle on ${Z_k (\tau)}$ is a pair $(E, s)$ where $E$ is a bundle on ${Z_k(\tau)}$ and $s = (s_1, s_2)$ is a trivialization of $E$ over ${Z_k^\circ (\tau)}$, that is, for each point $p \in {Z_k^\circ(\tau)}$ the vectors $s_1 (p)$ and $s_2 (p)$ are linearly independent. 
\item A {\it framed} rank~$2$ bundle on $F_m$ (resp.\ on $N$) is a pair $(E', f)$ where $E'$ is a bundle on $F_m$ (resp.\ on $N$) and $f = (f_1, f_2)$ is a trivialization of $E'$ over $N^\circ$, that is, for each point $p \in N^\circ$ the vectors $f_1 (p)$ and $f_2 (p)$ are linearly independent. 
\end{enumerate}
\end{definition}

Given framed bundles $(E, s)$ on $Z_k(\tau)$ and $(E', f)$ on $N$ we obtain a framed bundle $(E, s) \cup_{s = f} (E', f)$ on $F_m$ by identifying the framings, {\it i.e.}\ by setting $(s_1 (p), s_2 (p)) = (f_1 (p), f_2 (p))$ for all $p \in N^\circ$. (This procedure is called {\it holomorphic patching} in \cite{gasparim3}.) In order to describe instantons on $Z_k (\tau)$ we consider the case when $E'$ is the trivial bundle over $N$, because the charge has no contribution from a neighbourhood of infinity. We denote the resulting bundle on $F_m$ by $\overbar E$. Note that then $E$ extends trivially to $\overbar E$ in the terminology of Definition \ref{extendstrivially}, but the trivialization over $N^\circ$ is extra data.

\begin{lemma}
\label{split}
Consider $E = \mathcal O_{Z_k(\tau)}(j) \oplus \mathcal O_{Z_k(\tau)}(-j)$ together with a framing $s$ on ${Z_k^\circ(\tau)}$ and let $\overbar E$ be a framed extension of $E$ to the Hirzebruch surface $F_m$ as just described. Then $\overbar E$ is a split bundle over $F_m$.
\end{lemma}

\begin{proof} We write 
\[
0 \totikz \mathcal O_{Z_k(\tau)}(j) \toarg{i} E \totikz \mathcal O_{Z_k(\tau)}(-j)\rightarrowtikz 0
\]
and denote by $L$ the image of $\mathcal O_{Z_k(\tau)}(j)$ by $i$ inside $E$. We claim that $L$ extends to a line bundle on  $F_m$.
Since $E$ is split, we can choose the framing $s$ to be such that $s_1(p) \in L$ for all $p \in {Z_k^\circ(\tau)}$. Now, just take the trivial line bundle over $N$ framed by $f_1$ over $N^\circ$ and extend $L$ by identifying $s_1 = f_1$.

Inverting the sequence, we write 
\[
0 \totikz \mathcal O_{Z_k(\tau)}(-j) \toarg{t} E \totikz \mathcal O_{Z_k(\tau)}(j)\rightarrowtikz 0
\]
and now denote by $L^{-1}$ the image of $\mathcal O_{Z_k(\tau)}(-j)$ by $t$ inside $E$. Extend $L^{-1}$ to $F_m$ by choosing a framing $s_2$ and identifying to the framing $f_2$ over $N^\circ$ of a trivial line bundle over $N$. 
Since $E$ is split over ${Z_k(\tau)}$ we can choose $s_1$ and $s_2$ such that the  pair $(s_1,s_2)$ is linearly independent, hence 
it is a framing of the trivial bundle over $N^\circ$. Because $N^\circ$ is open and connected, any two framings of the trivial bundle over $N^\circ$ are related by a change of coordinates. So, the fact that the bundles $L$ and $L^{-1}$ extend to $F_m$ does not depend on the choice of framings. We conclude that $\overbar E$ splits. 
\end{proof}

\begin{theorem}
\label{empty}
Let $Z_k (\tau)$ be a nontrivial deformation of $Z_k$. Then the moduli space of (framed) irreducible $\operatorname{SU} (2)$-instantons on $Z_k (\tau)$ is empty.
\end{theorem}

\begin{proof}
Deformations of $Z_k$ extend to deformations of its compactification $F_k \supset Z_k$, see \S \ref{relationtohirzebruch}. The Hirzebruch surface $F_k$ deforms to a lower Hirzebruch surface $F_m$, where $m < k$ (with $m \equiv k$ $\mathrm{mod}\;2$) depends on $\tau \in \H^1 (F_k, \mathcal T_{F_k}) \simeq \H^1 (Z_k, \mathcal T_{Z_k})$. We choose $D$ to be the complement of $Z_k (\tau) \subset F_m$. Under the Kobayashi--Hitchin correspondence for $F_m$, the moduli space of irreducible $\operatorname{SU} (2)$-instantons is in one-to-one correspondence with stable holomorphic bundles. To find instantons on $Z_k (\tau)$ we are thus looking for stable bundles on $F_m$ which are trivial on $D$. By Theorem \ref{decomposable}, all holomorphic vector bundles on $Z_k (\tau)$ are split. But by Lemma \ref{split} framed bundles on $F_m$ which restrict to split bundles on $Z_k (\tau)$ are not stable.
\end{proof}

Informally, the fact that instantons disappear after a classical deformation may also be understood as a consequence of Theorem \ref{nocompactcurves}, since there is no compact curve to hold the local charge.

\begin{remark}
Even though the above line of argument seems to suggest that the behaviour of instanton moduli in families might be better behaved in the compact case, we stress that at least in the context of the instanton partition function \cite{Ne}, the noncompactness of the underlying surface is essential for the nontriviality of the theory.
\end{remark}

\paragraph{\bf Acknowledgements} We thank Barbara Fantechi and Pushan Majumdar for helpful discussions and the anonymous referee for the many detailed suggestions greatly improving the exposition of the paper. S.\,B.\ would like to thank the Studienstiftung des deutschen Volkes for support. Part of this paper was written while S.\,B.\ was visiting the Department of Physics of UNAB in Santiago and we thank Per Sundell for his hospitality and support under CONICYT grant DPI 20140115. E.\,G.\ was partially supported by a Simons Associateship grant of ICTP and Network grant NT8 of the Office of External Activities, ICTP, Italy.

\end{document}